\documentclass[authoryear,11pt, a4paper]{article}

\usepackage[left = 2.5cm, right = 2.5cm, top = 3cm, bottom = 3cm]{geometry}
\usepackage{natbib}
\usepackage[english]{babel}
\usepackage{amsmath}
\usepackage{amssymb}
\usepackage{latexsym}

\usepackage{graphicx}
\usepackage{subfigure}
\usepackage{color}

\usepackage[colorlinks,citecolor=black,urlcolor=black]{hyperref}

\newcommand{\G}{\mathcal{G}}% graph
\newcommand{\E}{\mathcal{E}}% edge set
\newcommand{\Pt}{\pi}% single path
\newcommand{\PT}{\Pi}% set of paths
\newcommand{\lp}{\langle}%left delimiter for paths
\newcommand{\rp}{\rangle}%right delimiter for paths
\newcommand{\diag}{\operatorname{diag}}
\newcommand{\IF}{\operatorname{IF}}
\newcommand{\infCORR}[1]{\Omega^{#1}}
\newcommand{\given}{\cdot}% symbol in subscript for "partial"
\newcommand{\vc}[2]{(#1)(#2)}% subscript of vector correlation and vector alienation coefficient
\newcommand{\w}{\omega}% symbol for weight

\usepackage{amsthm}
\theoremstyle{plain}% default
\newtheorem{theorem}{Theorem}[section]
\newtheorem{lemma}[theorem]{Lemma}
\newtheorem{proposition}[theorem]{Proposition}
\newtheorem{corollary}[theorem]{Corollary}

\theoremstyle{definition}

\newtheorem{example}{Example}[section]

\theoremstyle{remark}

\title{Path Weights in Concentration Graphs\\[1ex]}

\author{Alberto Roverato\\\texttt{alberto.roverato@unipd.it}\\ Dept. of Statistical Sciences \\ Universit\`a di Padova \\ Padova, Italy
\and Robert Castelo\\ \texttt{robert.castelo@upf.edu}\\ Dept. of Experimental and Health Sciences \\ Universitat Pompeu Fabra \\ Barcelona, Spain\\}
\date{April 28, 2020}
\begin{document}

\maketitle
\begin{abstract}
A graphical model provides a compact and efficient representation of the association structure of a multivariate distribution by means of a graph. Relevant features of the distribution are represented by vertices, edges and other higher-order graphical structures, such as cliques or paths. Typically, paths play a central role in these models because they determine the independence relationships among variables. However, while a theory of path coefficients is available in models for directed graphs, little has been investigated about the strength of the association represented by a path in an undirected graph.  Essentially, it has been shown that the covariance between two variables can be decomposed into a sum of weights associated with each of the paths connecting the two variables in the corresponding concentration graph. In this context, we consider concentration graph models and provide an extensive analysis of the properties of path weights and their interpretation. More specifically, we give an interpretation of covariance weights through their factorisation into a partial covariance and an inflation factor. We then extend the covariance decomposition over the paths of an undirected graph to other measures of association, such as the marginal correlation coefficient and a quantity that we call the inflated correlation. The usefulness of these findings is illustrated with an application to the analysis of dietary intake networks.
\end{abstract}


\section{Introduction}
%---------------------------------------------------------
Statistical models associated with a graph, called graphical models, are of significant interest in many modern applications and have become a popular tool for representing network structures in applied contexts such as genetic and brain network analysis; see \citep{maathuis2019handbook} for a recent review of the state of art of graphical models.

In graphical models, paths joining vertices of the graph are the main tools used in the definition of separation criteria and therefore, in the specification of the association structure of the variables. In models for continuous acyclic directed graphs it is also possible to compute path coefficients. Specifically, the theory of path analysis, originated from the seminal work of \citet{wright1921correlation}, provides a way of quantifying the relative importance of causal relationships represented by directed paths. On the other hand, in models for undirected graphs little has been investigated on the strength of the association encoded by paths. A prominent exception is the work by \citet{jones2005covariance}, which introduced a covariance decomposition in terms of additive weights associated with undirected paths in concentration graph models. A concentration graph \citep{cox1996multivariate,whittaker1990graphical,lauritzen1996graphical} is a labelled undirected graph that provides a representation of the sparsity pattern of an inverse covariance matrix. In such a graph, labelled vertices are associated with random variables and missing edges with zero entries in the inverse covariance matrix. The interpretation of the path weights of \citet{jones2005covariance}, however, remained unexplored until \citet{roverato2017networked} provided one for the basic case of single-edge paths. Some preliminary results on the paths between vertices in a tree were also given by \citet{pmlr-v72-roverato18a}.

In this article, we consider concentration graph models and provide an extensive analysis of the properties of the weight of an arbitrary path and of its interpretation. We first consider the original weights introduced by \citet{jones2005covariance} obtained from the decomposition of the covariance between two variables, and then move on to a more general setting that also includes the weights obtained from the decomposition of certain normalised measures of association. More specifically, we consider the decomposition of the correlation coefficient between two variables and of a novel, to the best of our knowledge, normalized measure of linear association that we call the inflated correlation coefficient. We find that, as far as path weights are concerned, the latter quantity should be preferred to the traditional correlation coefficients, because the corresponding weights admit a clearer interpretation. Furthermore, the weights relative to inflated correlations satisfy a set of properties which are consistent with the properties of weights computed from covariance coefficients. Finally, we illustrate possible uses of the theory of path weights in  applied contexts by the analysis of food intake patterns based on two, sex-specific, dietary intake networks from \citet{iqbal2016}.

\section{Background and notation}\label{SEC:notation}
%-------------------------------------------------
\subsection{Concentration graphs}
%---------------------------------------------------------
An undirected graph with vertex set $V$ is a pair $\G=(V, \E)$ where $\E$ is a set of edges, which are unordered pairs of vertices; formally $\E\subseteq V\times V$. The graphs we consider have no self-loops, that is $\{v, v\}\not\in \E$ for any $v\in V$. We say that a graph $\G^{\prime}=(V^{\prime}, \E^{\prime})$ is a subgraph of $\G$ if $V^{\prime}\subseteq V$ and $\E^{\prime}\subseteq \E$. The subgraph of $\G$ induced by $A\subseteq V$ is the undirected graph $\G_{A}$ with vertex set $A$ and edges $\E_{A}=\{\{u, v\}\in \E: u,v\in A\}$. A path of length $k\geq 2$ between $x$ and $y$ in  $\G$ is a sequence $\Pt=\lp x=v_{1},\ldots,v_{k}=y\rp$ of  distinct vertices such that $\{v_{i},v_{i+1}\}\in \E$ for every $i=1,\ldots,k-1$. We denote by $V(\Pt)\subseteq V$ and $\E(\Pt)\subseteq \E$ the set of vertices and edges of the path $\pi$, respectively. We write $\Pt_{xy}$ when we want to make more explicit which are the endpoints of the path and, furthermore, when clear from the context, we will write $P\equiv V(\Pt)$ to improve the readability of sub- and super-scripts. For a pair of vertices $x,y\in A$, we denote by $\PT^{V}_{xy}\equiv\PT_{xy}$ the collection of all paths between $x$ and $y$ in $\G$ and by $\PT^{A}_{xy}$ the set of paths joining the same pair of vertices in $\G_{A}$. It is straightforward to see that $\PT^{A}_{xy}\subseteq \PT_{xy}$ and, more specifically,  $\Pt\in \PT^{A}_{xy}$ if and only if $\Pt$ is such that $\Pt\in\PT_{xy}$ and $V(\Pt)\subseteq A$.

Let $X\equiv X_{V}$ be a random vector indexed by a finite set $V=\{1,\ldots,p\}$ so that for $A\subseteq V$, $X_{A}$ is the subvector of $X$ indexed by $A$. The random vector $X_{V}$ has probability distribution $\mathbb{P}_{V}$ and  covariance matrix $\Sigma=\{\sigma_{uv}\}_{u,v\in V}$.  The concentration  matrix $K=\{\kappa_{uv}\}_{u,v\in V}$ of $X_{V}$ is the inverse of its covariance matrix, that is $K=\Sigma^{-1}$. We say that $K$ is adapted to a graph $\G=(V, \E)$ if for every $\kappa_{uv}\neq 0$, with $u\neq v$, it holds that $\{u,v\}\in \E$ and, accordingly, we call $\G$ a concentration graph of $X_{V}$. The concentration graph model \citep{cox1996multivariate}  with graph $\G=(V, \E)$ is the family of multivariate normal distributions whose concentration matrices are adapted to $\G$. The latter model has also been called a covariance selection model \citep{dempster1972covariance} and a graphical Gaussian model \citep{whittaker1990graphical}; we refer the reader to \citet{lauritzen1996graphical} for details and discussion.

For $A,B\subseteq V$ with $A\cap B=\emptyset$, the partial covariance matrix $\Sigma_{AA\given B}=\Sigma_{AA}-\Sigma_{AB}\Sigma_{BB}^{-1}\Sigma_{BA}$ is the covariance matrix of $X_{A}|X_{B}$, that is, the residual vector deriving from the linear least square predictor of $X_{A}$ on $X_{B}$ \citep[see][p.~134]{whittaker1990graphical}. We denote by $\sigma_{uv\given B}$, for $u,v\in A$, the entries of $\Sigma_{AA\given B}$ and recall that, in the Gaussian case,  $\Sigma_{AA\given B}$ coincides with the covariance matrix of the conditional distribution of $X_{A}$ given $X_{B}$. Note that we use the convention that $\Sigma_{AA}^{-1}=(\Sigma_{AA})^{-1}$ and, similarly, $\Sigma_{AA\given B}^{-1}=(\Sigma_{AA\given B})^{-1}$. Furthermore, if $\bar{A}=V\setminus A$ is the complement of a subset $A$ with respect to $V$, then it follows from the rule for the inversion of a partitioned matrix that $\Sigma_{AA\given \bar{A}}=K_{AA}^{-1}$ and, similarly, $\Sigma_{AA}^{-1}=K_{AA\given \bar{A}}$. A useful consequence of the latter equalities is that if $K$ is adapted to $\G$, then the concentration matrix of $X_{A}|X_{\bar{A}}$, i.e., $K_{AA}$, is adapted to $\G_{A}$, and therefore, $\G_{A}$ is a concentration graph of $X_{A}|X_{\bar{A}}$.

\subsection{Covariance decomposition over $\G$}
%-------------------------------------------------
In the analysis of graphical Gaussian models, \citet{jones2005covariance}  showed that the covariance between two variables can be computed as the sum of weights associated with the paths joining the two variables.
\begin{theorem}[\citet{jones2005covariance}]\label{THM:jones.west}
Let $K=\Sigma^{-1}$ be the concentration matrix of  $X_{V}$. If $K$ is adapted to the graph $\G=(V, \E)$ then for every $x,y\in V$  it holds that
\begin{align}\label{EQN:in.thm.jones.west}
\sigma_{xy}
    =\sum_{\Pt\in \PT_{xy}} \w(\Pt, \Sigma)
\end{align}
where
\begin{align}\label{EQN:jones.and.west.path.weight1}
     \w(\Pt, \Sigma)
     =  (-1)^{|P|+1}\;\frac{|K_{\bar{P} \bar{P}}|}{|K|}\;\prod_{\{u,v\}\in \E(\Pt)} \kappa_{uv}.
\end{align}
\end{theorem}
The quantity $\w(\Pt, \Sigma)$ in (\ref{EQN:jones.and.west.path.weight1}) represents the contribution of the path $\Pt$ to the covariance $\sigma_{xy}$ and, for this reason, we call it the covariance weight of $\Pt$ relative to $X_{V}$. More generally, we will refer to (\ref{EQN:in.thm.jones.west}) as the covariance decomposition over $\G$.  There is no clear interpretation associated with the value taken by a covariance weight, with the exception of paths consisting of a single edge  \citep{roverato2017networked}, and the forthcoming sections will address this issue. We also recall that another interesting decomposition of the covariance in Gaussian models, in terms of walk-weights, can be found in \citet{malioutov2006walk} and references therein. Unlike paths, walks can cross an edge multiple times.

\section{Inflation factors}\label{SEC:inflation.factors}
%-------------------------------------------------
Linear regression diagnostics use a quantity called  the variance inflation factor to quantify the effect of multicollinearity, and this section deals with a generalized version of the inflation factor that arises naturally in the theory of path weights developed in this paper.

The variance inflation factor of $X_{v}$ on $X_{V\setminus \{v\}}$ is defined as $\IF_{v}=1/(1-\rho^{2}_{\vc{v}{V\setminus \{v\}}})$  where $\rho_{\vc{v}{V\setminus \{v\}}}$ is the multiple correlation of $X_{v}$ on $X_{V\setminus \{v\}}$. $\IF_{v}$ equals 1 when $X_{v}$ and $X_{V\setminus \{v\}}$ are uncorrelated so that $\rho_{\vc{v}{V\setminus \{v\}}}=0$; otherwise $\IF_{v}>1$ and its value increases as  $\rho_{\vc{v}{V\setminus \{v\}}}$ increases \citep[see][]{belsley2005regression,chatterjee2015regression}.

\citet{fox1992generalized} considered the case where one is interested in sets of regressors rather than individual regressors and introduced a generalized version of the variance inflation factor; specifically, for a pair of subsets $A,B\subseteq V$, with $A\cap B=\emptyset$ this is given by,
\begin{align}\label{EQN:PWIF}
  \IF_{A}^{B}=\frac{|\Sigma_{AA}||\Sigma_{BB}|}{|\Sigma_{A\cup B A\cup B}|},
\end{align}
so that $\IF_{A}^{B}=\IF_{v}$ when $A=\{v\}$ and $B=V\setminus \{v\}$. We will refer to $\IF_{A}^{B}$ as the inflation factor of $A$ on $B$, and to simplify the notation we will write $\IF_{A}$ when $B=V\setminus A$.
It is also worth remarking that the covariance matrices we consider are assumed to be positive definite so that $\IF_{A}^{B}\geq 1$ and we set $\IF_{A}^{B}=1$ whenever either $A=\emptyset$ or $B=\emptyset$. More generally, throughout this paper we use the convention that the determinant of a submatrix whose rows and columns are indexed by the empty set is equal to 1.
The following lemma provides some ways to compute $\IF_{A}^{B}$ alternative to (\ref{EQN:PWIF}).
\begin{lemma}\label{THM:computation.of.IF}
Let $\Sigma$ the covariance matrix of $X_{V}$, then for any pair of  subsets $A,B\subseteq V$, with $A\cap B=\emptyset$, it holds that
\begin{align}\label{EQN:lemma.comp.IF.01}
\IF_{A}^{B}=\frac{|\Sigma_{AA}|}{|\Sigma_{AA\given B}|}=\frac{|\Sigma_{BB}|}{|\Sigma_{BB\given A}|}.
\end{align}
Furthermore,
\begin{align}\label{EQN:lemma.comp.IF.02}
\IF_{A}^{B}
=\frac{|\Sigma_{A\cup B A\cup B}|}{|\Sigma_{AA\given B}||\Sigma_{BB\given A}|},
\end{align}
that in the special case where  $B=\bar{A}$ gives $\IF_{A}=(|K_{AA}|\times |K_{\bar{A}\bar{A}}|)/|K|$,
where $K=\Sigma^{-1}$.
\end{lemma}
\begin{proof}
From the Schur's determinant identities $|\Sigma_{A\cup B A\cup B}|=|\Sigma_{AA\given B}||\Sigma_{BB}|$ and $|\Sigma_{A\cup B A\cup B}|=|\Sigma_{BB\given A}||\Sigma_{AA}|$ one has that
\begin{align*}
\IF_{A}^{B}
=\frac{|\Sigma_{AA}||\Sigma_{BB}|}{|\Sigma_{A\cup B A\cup B}|}
=\frac{|\Sigma_{AA}||\Sigma_{BB}|}{|\Sigma_{AA\given B}||\Sigma_{BB}|}
=\frac{|\Sigma_{AA}|}{|\Sigma_{AA\given B}|},
\end{align*}
and in a similar way one can show that $\IF_{A}^{B}=|\Sigma_{BB}|/|\Sigma_{BB\given A}|$. From the latter identity one obtains that
\begin{align}\label{EQN:alt.IF.proof001}
\IF_{A}^{B}
=\frac{|\Sigma_{BB}|}{|\Sigma_{BB\given A}|}
=\frac{|\Sigma_{BB}||\Sigma_{AA\given B}|}{|\Sigma_{BB\given A}||\Sigma_{AA\given B}|}
= \frac{|\Sigma_{A\cup B A\cup B}|}{|\Sigma_{BB\given A}||\Sigma_{AA\given B}|},
\end{align}
which gives (\ref{EQN:lemma.comp.IF.02}). Finally, in the case where $B=\bar{A}$, then (\ref{EQN:alt.IF.proof001}) can be written as
\begin{align*}
\IF_{A}
=\frac{|\Sigma|}{|\Sigma_{AA\given\bar{A}}||\Sigma_{\bar{A}\bar{A}\given A}|}
=\frac{|\Sigma_{AA\given\bar{A}}^{-1}||\Sigma_{\bar{A}\bar{A}\given A}^{-1}|}{|\Sigma^{-1}|}
=\frac{|K_{AA}||K_{\bar{A}\bar{A}}|}{|K|},
\end{align*}
and this completes the proof.
\end{proof}

If we denote by $\Omega$ the correlation matrix of $X_{V}$, then one can write the inflation factor  in (\ref{EQN:PWIF}) as $\IF_{A}^{B}=(|\Omega_{AA}|\times|\Omega_{BB}|)/|\Omega_{A\cup B,A\cup B}|$. The determinant of $|\Omega|$ is a common global measure of collinearity usually justified by noting that $|\Omega|=1$ for uncorrelated variables and $|\Omega|=0$ for perfectly collinear variables. As remarked by \citet{fox1992generalized}, this suggests the following interpretation of $\IF_{A}^{B}$: the inflation factor of $A$ on $B$ represents the global collinearity of $X_{A\cup B}$ scaled by the product of the collinearity internal to each of $X_{A}$ and $X_{B}$. Based on this perspective, \citet{fox1992generalized} suggested a generalization of (\ref{EQN:PWIF}) to the case where $X_{V}$ is partitioned into $k$ sets, $A_{1},\ldots A_{k}$, given by $(|\Sigma_{A_{1}A_{1}}|\times\cdots\times |\Sigma_{A_{k}A_{k}}|)/|\Sigma|$. The latter is consistent with the measure of collinearity provided by $|\Omega|$ because in the special case where every set contains a single variable one obtains,
\begin{align}\label{EQN:gen.IF.variances}
\frac{\prod_{v\in V}\sigma_{vv}}{|\Sigma|}=\frac{1}{|\Omega|}.
\end{align}
It is of interest for us to compare the different formulations of the inflation factor provided by (\ref{EQN:PWIF}) and (\ref{EQN:lemma.comp.IF.02}). Specifically, we note that equation (\ref{EQN:lemma.comp.IF.02}) allows for an alternative interpretation of $\IF_{A}^{B}$. In linear regression the strength of the linear association between variables is commonly measured by comparing the variability of the response variable with the variability of the same variable after it is linearly adjusted for the covariates. As shown in (\ref{EQN:lemma.comp.IF.01}), also the inflation factor is feasible of this type of interpretation, and equation (\ref{EQN:lemma.comp.IF.02}) writes the ratio of variability measures in a way that is symmetrical with respect to the two sets $A$ and $B$. The determinant of $\Sigma$ is the generalized variance of $X_{V}$, whereas $|\Sigma_{AA\given B}|$ and $|\Sigma_{BB\given A}|$ are the generalized variances of $X_{A}|X_{B}$ and $X_{B}|X_{A}$, respectively. Thus, one can interpret $\IF_{A}^{B}$ as the variability of $X_{V}$ scaled by the product of the variabilities internal to each of $X_{A}|X_{B}$ and $X_{B}|X_{A}$. This allows a generalization of the inflation factor to the case where $X_{V}$ is partitioned into $k$ sets as $|\Sigma|/(|\Sigma_{A_{1}A_{1}\given \bar{A}_{1}}|\times\cdots\times |\Sigma_{A_{k}A_{k}\given\bar{A}_{k}}|)$. The case where every set contains a single variable provides a global measure of collinearity for $X_{V}$ alternative to (\ref{EQN:gen.IF.variances}),
\begin{align}\label{EQN:gen.IF.partial.variances}
\frac{|\Sigma|}{\prod_{v\in V}\sigma_{vv\given V\setminus \{v\}}},
\end{align}
which plays a central role on the interpretation of the information provided by path weights.

To further clarify the connection existing between  (\ref{EQN:gen.IF.variances})  and (\ref{EQN:gen.IF.partial.variances}), we note that Hadamard's inequality implies both that $|\Sigma|\leq \prod_{v\in V}\sigma_{vv}$ and that $|K|\leq \prod_{v\in V}\kappa_{vv}$. In turn, this implies that $\prod_{v\in V}\sigma_{vv\given V\setminus \{v\}}\leq |\Sigma|\leq \prod_{v\in V}\sigma_{vv}$, because $\kappa_{vv}=1/\sigma_{vv\given V\setminus\{v\}}$ for every $v\in V$ \citep[][Corollary~5.8.1]{whittaker1990graphical}. In this way we can see that (\ref{EQN:gen.IF.variances}) scales $|\Sigma|$ with its upper bound whereas (\ref{EQN:gen.IF.partial.variances}) scales $|\Sigma|$ with its lower bound, and both quantities are global measures of linear association that can be regarded as inflations factors because they take values in the interval $[1,\infty)$.

\section{The inflated correlation matrix and other relevant matrices}\label{SEC:inflated.correlation}
%---------------------------------------------------------------------
We introduce here some scaled versions of both the covariance and the concentration matrix. We first scale $K$ to have unit diagonal and write $(I-R)=\diag(K)^{-\frac{1}{2}}K\diag(K)^{-\frac{1}{2}}$, where $I$ denotes the $p\times p$ identity matrix and $R$ is a matrix with partial correlations as off-diagonal entries and zeros on the main diagonal. This follows from the fact that, if we denote by $\rho_{xy\given A}$ the partial correlation coefficient of  $X_{x}$ and $X_{y}$ given $X_{A}$ then, for every $x,y\in V$ it holds that \citep[see][p.~130]{lauritzen1996graphical},
\begin{align*}%\label{EQN:partial.correlation.from.K}
\rho_{uv\given V\setminus \{u, v\}}
=\frac{-\kappa_{uv}}{\sqrt{\kappa_{uu}\kappa_{vv}}}
=\{R\}_{uv}\,.
\end{align*}

The correlation matrix $\Omega=\diag(\Sigma)^{-\frac{1}{2}}\Sigma\diag(\Sigma)^{-\frac{1}{2}}$ is computed by scaling $\Sigma$ to have unit diagonal. In this paper a central role is played by a different scaled version of the covariance matrix, which we denote by $\infCORR{V}$, and that can be computed as the inverse of $(I-R)$,
\begin{align}\label{EQN:def.Rinv}
\infCORR{V}=(I-R)^{-1}=\diag(K)^{\frac{1}{2}}\Sigma\diag(K)^{\frac{1}{2}}.
\end{align}
We call $\infCORR{V}$ the inflated correlation matrix because its entries, denoted by $\varrho^{V}_{uv}$ for $u,v\in V$, can be interpreted as inflated correlation coefficients. Furthermore, the determinant of $\infCORR{V}$ is the global measure of linear association between the variables in $X_{V}$ given in (\ref{EQN:gen.IF.partial.variances}).
\begin{lemma}\label{THM:interpretation.entries.Rinv}
The entries of $\infCORR{V}=\{\varrho^{V}_{uv}\}_{u,v\in V}$ are given by $\varrho_{vv}^{V}=\IF_{v}$ for every $v\in V$, whereas for every $u,v\in V$ with $v\neq u$,
\begin{align}\label{EQN:off-diag-Rinv}
    \varrho_{uv}^{V}
    =\rho_{uv}
    \times \sqrt{\IF_{u} \times \IF_{v}}\,,
\end{align}
so that for $|V|=2$ the inflated correlation $\varrho_{uv}^{V}$ in (\ref{EQN:off-diag-Rinv}) becomes $\varrho^{\{u,v\}}_{uv}=\rho_{uv}/(1-\rho_{uv}^{2})$.
Furthermore,
\begin{align*}%\label{EQN:det.omegaIF}
|\infCORR{V}|=\frac{|\Sigma|}{\prod_{v\in V}\sigma_{vv\given V\setminus\{v\}}},
\end{align*}
with $|\infCORR{V}|\geq 1$ and  $|\infCORR{V}|= 1$ if and only if $\Sigma$ is diagonal.
\end{lemma}
\begin{proof}
The entries of $\infCORR{V}=\diag(K)^{\frac{1}{2}}\Sigma\diag(K)^{\frac{1}{2}}$ can be computed as $\varrho^{V}_{uv}=\sigma_{uv}\sqrt{\kappa_{uu}\kappa_{vv}}$, and because $\kappa_{uu}=1/\sigma_{uu\given V\setminus\{u\}}$ and $\kappa_{vv}=1/\sigma_{vv\given V\setminus\{v\}}$ it follows that for every $u,v\in V$,
\begin{align}\label{EQN:proof.varrho}
  \varrho^{V}_{uv}
  =\frac{\sigma_{uv}}{\sqrt{\sigma_{uu\given V\setminus \{u\}}\sigma_{vv\given V\setminus \{v\}}}}
  =\frac{\sigma_{uv}}{\sqrt{\sigma_{uu}\sigma_{vv}}}
  \sqrt{\frac{\sigma_{uu}}{\sigma_{uu\given V\setminus \{u\}}}}
  \sqrt{\frac{\sigma_{vv}}{\sigma_{vv\given V\setminus \{v\}}}}.
\end{align}
For $u=v$, equation (\ref{EQN:proof.varrho}) becomes  $\varrho^{V}_{uu}=\sigma_{uu}/\sigma_{uu\given V\setminus \{u\}}$ so that, by Lemma~\ref{THM:computation.of.IF}, $\varrho^{V}_{uu}=\IF_{v}$. Similarly, for $u\neq v$ equation (\ref{EQN:proof.varrho}) and Lemma~\ref{THM:computation.of.IF} give (\ref{EQN:off-diag-Rinv}), as required. For $V=\{u,v\}$ it holds that $\IF_{u}=\IF_{v}=1/(1-\rho_{uv}^{2})$ so that (\ref{EQN:off-diag-Rinv}) can be written as $\varrho^{\{u,v\}}_{uv}=\rho_{uv}/(1-\rho_{uv}^{2})$.

The determinant of $\infCORR{V}$ can be computed from (\ref{EQN:def.Rinv}) as
\begin{align*}
|\infCORR{V}| = |\diag(K)||\Sigma|= |\Sigma|\prod_{v\in V}\kappa_{vv}=
\frac{|\Sigma|}{\prod_{v\in V}\sigma_{vv\given V\setminus\{v\}}}\,,
\end{align*}
whereas to show that $|\infCORR{V}|\geq 1$ it is sufficient to notice that $|\infCORR{V}|=1/|I-R|$ and that $|I-R|\leq 1$ by Hadamard's inequality. Furthermore, $|I-R|=1$ if and only if $R=0$, that is, if and only if  $K$, or equivalently $\Sigma$, is  diagonal.
\end{proof}
Hence, the off-diagonal entry $\varrho^{V}_{uv}$ of $\infCORR{V}$ is the correlation $\rho_{uv}$ inflated by a factor that is the product of the square root of the  inflation factors of $X_{u}$ and $X_{v}$ on $X_{V\setminus\{u\}}$ and $X_{V\setminus\{v\}}$, respectively. It is worth remarking that, unlike $\rho_{uv}$, which is a measure of marginal association, the computation of $\varrho_{uv}^{V}$ involves the joint distribution of $X_{V}$. Furthermore, it can be shown that $\varrho^{V}_{uv}\geq \varrho^{A}_{uv}$ for every $u,v\in A\subseteq V$ because $\IF_{v}^{V\setminus \{v\}}\geq \IF_{v}^{A\setminus \{v\}}$; see also \citet{pmlr-v72-roverato18a}. It is also interesting to recall that if the spectral radius of $R$ is smaller then 1, then $\lim_{k\rightarrow \infty} R^{k}=0$ and $\infCORR{V}=(I-R)^{-1}=I+R+R^{2}+\cdots$.

In the following we will also consider both correlations and inflated correlations computed with respect to the distribution of $X_{A}|X_{\bar{A}}$. The corresponding matrices are given by $\Omega_{,[A| \bar{A}]}=\{\rho_{uv\given \bar{A}}\}_{u,v\in A}$ and $\infCORR{A}_{AA\given \bar{A}}=\{\varrho^{A}_{uv\given \bar{A}}\}_{u,v\in A}$, respectively, where the different types of subscript used in the notation reflect the different ways in which the two matrices are computed. Indeed, similarly to $\Sigma_{AA\given\bar{A}}$, the matrix $\infCORR{A}_{AA\given \bar{A}}$ can be computed as
$\infCORR{A}_{AA\given \bar{A}}=\infCORR{V}_{AA}-\infCORR{V}_{A\bar{A}}(\infCORR{V}_{\bar{A}\bar{A}})^{-1}\infCORR{V}_{\bar{A}A}$ so that $\infCORR{A}_{AA\given \bar{A}}=(I-R)_{AA}^{-1}$. On the other hand,  $\Omega_{[A| \bar{A}]}$ is obtained by scaling $\Sigma_{AA\given \bar{A}}$ so that, in general, $\Omega_{[A| \bar{A}]}\neq \Omega_{AA\given\bar{A}}$.

\section{Interpretation of covariance path weights}\label{SEC:decomposition.weight}
%----------------------------------------------------------------
Our approach to the interpretation of covariance weights relies on a key  factorization of these quantities. More specifically, we show that any covariance weight can be written as the product of  a partial covariance weight by an inflation factor. The analysis of these two components allows one to gain insight into the meaning of path weights and the kind of information they provide.

If $K$ is adapted to the graph $\G=(V, \E)$, and $\Pt$ is a path  between $x$ and $y$ in $\G$, then it follows that $\Pt$ is also a path in $\G_{A}=(A, \E_{A})$ for any $A$ such that $V(\Pt)\subseteq A\subseteq V$. On the other hand, the concentration matrix $K_{AA}=\Sigma_{AA\given \bar{A}}^{-1}$ of $X_{A}|X_{\bar{A}}$ is adapted to the subgraph $\G_{A}$, and therefore it makes sense to compute
\begin{align*}%\label{EQN:def.partial.weight}
  \w(\Pt, \Sigma_{AA\given \bar{A}})
  =  (-1)^{|P|+1}\;\frac{|K_{A\backslash P A\backslash P}|}{|K_{AA}|}\;\prod_{\{u,v\}\in \E(\Pt)} \kappa_{uv},
\end{align*}
which is the covariance weight of $\Pt$ with respect to the distribution of $X_{A}|X_{\bar{A}}$, i.e., the partial covariance weight of $\Pt$ relative to $X_{A}|X_{\bar{A}}$. An immediate consequence of Theorem~\ref{THM:jones.west} is that $\w(\Pt, \Sigma_{AA\given\bar{A}})$ represents the contribution of the path $\Pt$ to the partial covariance $\sigma_{xy\given\bar{A}}$.
\begin{corollary}\label{THM:jones.west.special.case1}
Let $K=\Sigma^{-1}$ be the concentration matrix of $X_{V}$. If $K$ is adapted to the graph $\G=(V, \E)$, then for  every $A\subseteq V$ and  $x,y\in A$, $x\neq y$, it holds that
\begin{eqnarray}\label{EQN:cor.jones.west.special.case1}
\sigma_{xy\given\bar{A}}
    &=&\sum_{\Pt\in \PT_{xy}^{A}} \w(\Pt, \Sigma_{AA\given\bar{A}}).
\end{eqnarray}
\end{corollary}
\begin{proof}
The result follows by applying Theorem~\ref{THM:jones.west} to the distribution of $X_{A}|X_{\bar{A}}$ and noticing that the set of paths $\Pt\in \PT_{xy}^{A}$ coincides with the set of all paths between $x$ and $y$ in $\G_{A}$.
\end{proof}
It is worth remarking that equation (\ref{EQN:cor.jones.west.special.case1}) can be regarded as a generalization of the covariance decomposition in (\ref{EQN:in.thm.jones.west}), as it coincides with the latter when $A=V$.

The theory of path weights we develop here relies on the existing connection between  the covariance weight and partial covariance weight of a path. Specifically, the following result  provides a rule to compute the partial covariance weight of a path $\Pt$  from the covariance weight of $\Pt$. An interesting consequence of it  is that the inflation factor arises naturally as an updating multiplicative constant.
\begin{theorem}\label{THM:two.components.of.path.weights}
If the matrix  $K=\Sigma^{-1}$ is adapted to the graph  $\G=(V, \E)$, then for any path $\Pt$ in $\G$ and subset $A\subseteq V$ such that $V(\Pt)\subseteq A$, it holds that
\begin{eqnarray*}
\w(\Pt, \Sigma)=\w(\Pt, \Sigma_{AA\given \bar{A}})\times \IF_{P}^{\bar{A}}
\end{eqnarray*}
where $P=V(\Pt)$. Furthermore,
\begin{eqnarray*}
   \mathrm{sgn}\left\{\w(\Pt, \Sigma)\right\}
   =
   \mathrm{sgn}\left\{\w(\Pt, \Sigma_{AA\given\bar{A}})\right\},
  \qquad
    |\w(\Pt, \Sigma)|
  \geq
    |\w(\Pt, \Sigma_{AA\given\bar{A}})|
\end{eqnarray*}
where, for nonzero weights, the latter inequality  is an equality if and only if $\IF_{P}^{\bar{A}}=1$.
\end{theorem}
\begin{proof}
Recall that $|\Sigma_{PP}|=|K_{\bar{P} \bar{P}}|/|K|$, whereas $|\Sigma_{PP\given\bar{A}}|=|K_{A\backslash P A\backslash P}|/|K_{AA}|$. Hence, we can write
\begin{align*}
\w(\Pt, \Sigma)
    &=  (-1)^{|P|+1}\;|\Sigma_{PP}|\;\prod_{\{u,v\}\in \E(\Pt)} \kappa_{uv}
    =  (-1)^{|P|+1}\;|\Sigma_{PP\given\bar{A}|}\;\prod_{\{u,v\}\in \E(\Pt)} \kappa_{uv} \times \frac{|\Sigma_{PP}|}{|\Sigma_{PP\given\bar{A}}|}\\
    &= \w(\Pt, \Sigma_{AA\given  \bar{A}})\times \IF_{P}^{\bar{A}},
\end{align*}
where the inflation factor is computed as in Lemma~\ref{THM:computation.of.IF}. The remaining statements follow immediately from the fact that $\IF_{P}^{\bar{A}}\geq 1$.
\end{proof}

Theorem~\ref{THM:two.components.of.path.weights} shows that for any $A\subseteq V$ such that $V(\Pt)\subseteq A$, the partial covariance weight of the path $\Pt$ relative to $X_{A}|X_{\bar{A}}$ can be obtained as the product of the covariace weight of $\Pt$ by a term that is a function of the linear association between the variables indexed by the vertices in the path  $P=V(\Pt)$  and the variables indexed by $\bar{A}$. Next, we notice that for a given path $\Pt$ one can compute a weight $\w(\Pt, \Sigma_{AA\given\bar{A}})$ for every $A$ such that $P\subseteq A\subseteq V$, thereby obtaining a collection of weights for $\Pt$. Furthermore, one can iteratively apply Theorem~\ref{THM:two.components.of.path.weights} to see that $|\w(\Pt, \Sigma_{PP\given\bar{P}})|\leq |\w(\Pt, \Sigma_{AA\given\bar{A}})|\leq |\w(\Pt, \Sigma)|$
for every $A$ such that $P\subseteq A\subseteq V$. Hence, $\w(\Pt, \Sigma_{PP\given\bar{P}})$ is the smallest partial covariance weight of $\Pt$ and the factorization
\begin{align}\label{EQN:weight.decoposition}
  \w(\Pt, \Sigma)&=\w(\Pt, \Sigma_{PP\given\bar{P}})\times \IF_{P}
\end{align}
plays a central role among all the possible decompositions of $\w(\Pt, \Sigma)$ implied by  Theorem~\ref{THM:two.components.of.path.weights}.
Equation (\ref{EQN:weight.decoposition}) shows that $\w(\Pt, \Sigma)$ is obtained by  multiplying the partial path weight $\w(\Pt, \Sigma_{PP\given \bar{P}})$ by the inflation factor $\IF_{P}$. We note that $\w(\Pt, \Sigma_{PP\given\bar{P}})$ and $\IF_{P}$ provide two clearly distinct pieces of information. Specifically:
\begin{enumerate}
  \item[(a)] $\w(\Pt, \Sigma_{PP\given\bar{P}})$ is the covariance weight of the path $\Pt$ linearly adjusted for all the variables not involved in the path. Hence, as far as linear relationships are concerned, this quantity provides no information on how the path $\Pt$ interacts with the variables in the rest of the network. This kind of interpretation is even stronger in the case where the variables are jointly Gaussian because  $\w(\Pt, \Sigma_{PP\given\bar{P}})$ is the weight of the path $\Pt$ computed in the conditional distribution of $X_{P}|X_{\bar{P}}$.
  \item[(b)] The inflation factor $\IF_{P}$ depends on $\Pt$ only through its vertex set $P=V(\Pt)$ and it is a measure of the strength of the association of the variables in the path with the remaining variables. For instance, when the variables indexed by $V(\Pt)$ are disconnected from the rest of the network, it holds that $\IF_{P}=1$ and therefore $\w(\Pt, \Sigma)=\w(\Pt, \Sigma_{PP\given\bar{P}})$.
\end{enumerate}

The factorization (\ref{EQN:weight.decoposition}) splits a covariance weight into two clearly distinct pieces of information. However,  while the value of the inflation factor has a clear interpretation, the partial covariance weight requires some additional consideration to clarify its meaning. If we apply Corollary~\ref{THM:jones.west.special.case1} with $A=P=V(\Pt)$ then we can write
\begin{eqnarray}\label{EQN:par.cov.sum}
  \sigma_{xy\given\bar{P}}=\sum_{\Pt^{\prime}\in \PT_{xy}^{P}} \w(\Pt^{\prime}, \Sigma_{PP\given\bar{P}}).
\end{eqnarray}
Thus, $\w(\Pt, \Sigma_{PP\given\bar{P}})$ in (\ref{EQN:weight.decoposition}) can be interpreted as the contribution of the path $\Pt$ to the partial covariance between $X_{x}$ and $X_{y}$ given all the variables that do not belong to the path. Hereafter, we consider some special cases of interest where the interpretation   of $\w(\Pt, \Sigma_{PP\given\bar{P}})$ is specially straightforward, whereas in the next section we address this issue with respect to normalized measures of association.

A first case of interest is when the weights of the paths in $\PT^{P}_{xy}$  are either all positive or all negative; recall that, by Theorem~\ref{THM:two.components.of.path.weights}, $\w(\Pt, \Sigma)$ has the same sign as any partial covariance weight of $\Pt$. In that case, (\ref{EQN:par.cov.sum}) meaningfully decomposes $\sigma_{xy\given\bar{P}}$ along the partial paths in $\PT^{A}_{xy}$, each contributing to a proportion of the latter covariance. In this context, a relevant role is played by the family of Gaussian totally positive distributions of order two. \citet{fallat2017total} and \citet{lauritzen2019maximum} studied the multivariate association structure of totally positive distributions of order two, and showed that they satisfy useful properties within the undirected graphical model framework. Interestingly, a Gaussian distribution is totally positive  of order two if and only if all the entries of $R$ are either zero or positive \citep{karlin1983m,fallat2017total}. It follows immediately that these distributions have the additional property that, for any path, both the weight and all of the partial weights are positive. More generally, a random variable $X_{V}$ has a signed totally positive distribution of order two if there exists a diagonal matrix $\Delta=\{\delta_{vv}\}_{v\in V}$ with $\delta_{vv}=\pm 1$ such that the distribution of $\Delta X_{V}$ is totally positive of order two \citep{karlin1981total,lauritzen2019maximum}. In the Gaussian case, for this wider family of distributions the weights of the paths can also be given a clear interpretation because it follows from Lemma~\ref{THM:phi.almost.scale.invariant} of Section~\ref{SEC:scale.transform} that in Gaussian signed totally positive distributions of order two, for any pair $x,y\in V$ the paths $\PT_{xy}$ are either all positive or all negative.

A second special case is when the path $\Pt$ has no chords. A chord of a path $\Pt$ is any edge joining a pair of nonadjacent vertices of $\pi$, that is $\{v_{i},v_{j}\}\in \E$ is a chord of $\pi$ if $\{v_{i},v_{j}\}\subseteq V(\pi)$ and $\{v_{i},v_{j}\}\not\in \E(\pi)$. A path is said to be chordless if it has no chords \citep{pelayo2013geodesic}, and we note that a path $\Pt$ between $x$ and $y$ is chordless if and only if $\PT_{xy}^{P}=\{\Pt\}$ and thus, because by (\ref{EQN:par.cov.sum}) it holds that $\w(\Pt, \Sigma_{PP\given\bar{P}})=\sigma_{xy\given\bar{P}}$, then (\ref{EQN:weight.decoposition}) takes the form,
\begin{align}\label{EQN:weight.chordless.path}
\w(\Pt, \Sigma)&=\sigma_{xy\given\bar{P}}\times \IF_{P}.
\end{align}
Equation (\ref{EQN:weight.chordless.path}) shows that the covariance weight of a chordless path  has a precise interpretation because it can be written as a straightforward combination of two well-established quantities: a partial covariance and an inflation factor. \citet{roverato2017networked} considered the path weights  (\ref{EQN:weight.chordless.path}) for the particular case of paths made up of a single edge. Every edge of the graph is a chordless path, and equation (\ref{EQN:weight.chordless.path}) extends the result of \citet{roverato2017networked} to arbitrary chordless paths.

Finally, we consider the case where $\Pt$ is the unique path between $x$ and $y$ in $\G$, that is, $\PT_{xy}=\{\Pt\}$. It follows immediately from (\ref{EQN:in.thm.jones.west}) that $\w(\Pt, \Sigma)=\sigma_{xy}$, but on the other hand $\PT_{xy}=\{\Pt\}$ implies $\PT^{P}_{xy}=\{\Pt\}$, i.e., $\Pt$ is a chordless path, so (\ref{EQN:weight.chordless.path}) applies, leading to a decomposition of the covariance into a partial covariance and an inflation factor,
\begin{align}\label{EQN:weight.unique.path}
\sigma_{xy}&=\sigma_{xy\given\bar{P}}\times \IF_{P}.
\end{align}
Equation (\ref{EQN:weight.unique.path}) allows one to clarify the mechanism by which the covariance $\sigma_{xy}$ updates its value when the distribution of $X_{\{x,y\}}$ is adjusted for $X_{\bar{P}}$. It implies, for instance, that $\sigma_{xy}$ and $\sigma_{xy\given\bar{P}}$  have the same sign and that $|\sigma_{xy}|\geq |\sigma_{xy\given\bar{P}}|$. Equation  (\ref{EQN:weight.unique.path}) cannot be applied to arbitrary paths; nevertheless there are relevant instances of models involving unique paths. More concretely, graphical models with a tree structure play  an important role due to their computational tractability; see, among others, \citet{edwards2010selecting}, \citet{choi2011learning}, \citet{lafferty2012sparse} and \citet[][Section~7.4]{hojsgaard2012graphical}. Since in a tree  $|\Pi_{xy}|=1$ for every $x,y\in V$ with $x\neq y$, then in concentration tree models the relationship (\ref{EQN:weight.unique.path}) holds true for all pairs of variables \citep{pmlr-v72-roverato18a}.

\section{Generalized path weights}\label{SEC:scale.transform}
%---------------------------------------------------------
\subsection{Correlation and inflated correlation path weights}
%---------------------------------------------------------
The work of \citet{jones2005covariance} focuses on the covariance matrix $\Sigma$ of $X_{V}$ and provides a decomposition of the entries of this matrix along the paths of a concentration graph $\G$ of $X_{V}$. We note that such a decomposition relies on the fact that $\Sigma^{-1}$ is adapted to $\G$, and therefore can be immediately applied to any positive-definite matrix $\Gamma=\{\gamma_{uv}\}_{u,v\in V}$ whose inverse $\Gamma^{-1}=\Theta=\{\theta_{uv}\}_{u,v\in V}$ is adapted to $\G$. Here, we consider an arbitrary matrix $\Gamma$ obtained as $\Gamma=\Delta\Sigma\Delta$ where $\Delta=\{\delta_{uv}\}_{u,v\in V}$ is a diagonal matrix with nonzero diagonal entries. Note that both $\Omega$ and $\infCORR{V}$ can be obtained in this way by setting $\Delta=\diag(\Sigma)^{-\frac{1}{2}}$ and $\Delta=\diag(K)^{\frac{1}{2}}$, respectively. Indeed, correlations and inflated correlations are both normalized measures of association, and for this reason it is of interest to compute and interpret the values of the corresponding weights.

There exist two possible ways to exploit  Theorem~\ref{THM:jones.west} to obtain a decomposition of an off-diagonal entry $\gamma_{xy}$ of $\Gamma$. The first is to replace $\Sigma$ by $\Gamma$ in (\ref{EQN:in.thm.jones.west}) and (\ref{EQN:jones.and.west.path.weight1}) thereby obtaining $\gamma_{xy}=\sum_{\Pt\in \PT_{xy}} \w(\Pt, \Gamma)$. On the other hand, one can consider the basic decomposition $\sigma_{xy}=\sum_{\Pt\in \PT_{xy}} \w(\Pt, \Sigma)$ in (\ref{EQN:in.thm.jones.west}) and then multiply both sides by $\delta_{xx}\delta_{yy}$ to obtain $\gamma_{xy}=\sum_{\Pt\in \PT_{xy}} \delta_{xx}\delta_{yy}\w(\Pt, \Sigma)$. The following  result shows that these two approaches are equivalent because they lead to identical decompositions.
\begin{lemma}\label{THM:phi.almost.scale.invariant}
If the matrix  $K=\Sigma^{-1}$ is adapted to the graph  $\G=(V, \E)$ and $\Gamma=\Delta\Sigma\Delta$ where $\Delta=\{\delta_{uv}\}_{u,v\in V}$ is a diagonal matrix with $\delta_{vv}\neq 0$ for all $v\in V$, then for any path $\Pt$ between $x$ and $y$ in $\G$ it holds that
$\w(\Pt, \Gamma)=\delta_{xx}\delta_{yy}\,\w(\Pt, \Sigma)$.
\end{lemma}
\begin{proof}
If $K=\Sigma^{-1}$ then   $(\Delta\Sigma\Delta)^{-1}=\Delta^{-1}K\Delta^{-1}$ and $|\Delta^{-1}K\Delta^{-1}|=\frac{|K|}{\prod_{v\in V}\delta^{2}_{vv}}$. Hence, we can compute  (\ref{EQN:jones.and.west.path.weight1}) as
\begin{align*}
  \w(\Pt, \Gamma)
  &=(-1)^{|P|+1}
  \frac{|\;\Delta^{-1}_{\bar{P}\bar{P}}K_{\bar{P}\bar{P}}\Delta^{-1}_{\bar{P}\bar{P}}|}{|\Delta^{-1}K\Delta^{-1}|}
  \prod_{\{u,v\}\in \E(\Pt)}\frac{\kappa_{uv}}{\delta_{uu}\delta_{vv}}\\
  &=
  \frac{|\Delta^{-1}_{\bar{P}\bar{P}}|^{2}
  }
  {
  |\Delta^{-1}|^{2}
  }\left(\prod_{\{u,v\}\in \E(\Pt)}\frac{1}{\delta_{uu}\delta_{vv}}\right)
  \,\w(\Pt, \Sigma)\\
  &=\frac
  {
  |\Delta|^{2}
  }
  {
  |\Delta_{\bar{P}\bar{P}}|^{2}
  \prod_{\{u,v\}\in \E(\Pt)}\delta_{uu}\delta_{vv}
  }\,\w(\Pt, \Sigma)\\
  &=\frac{\prod_{v\in V} \delta^{2}_{vv}}{\left(\textstyle\prod_{v\in \bar{P}} \delta_{vv}^{2}\right) \left(\delta_{xx}\delta_{yy}\textstyle\prod_{v\in P\setminus \{x,y\}} \delta^{2}_{vv}\right)}\,\w(\Pt, \Sigma)\\
  &=\frac{\prod_{v\in V} \delta^{2}_{vv}}{\delta_{xx}\delta_{yy}\textstyle\prod_{v\in V\setminus \{x,y\}} \delta^{2}_{vv}}\,\w(\Pt,\Sigma)
  =\delta_{xx}\delta_{yy}\,\w(\Pt, \Sigma).
\end{align*}
\end{proof}
Lemma~\ref{THM:phi.almost.scale.invariant} allows one to extend the properties of covariance path weights to the weights of other measures of association. More specifically, we show below that results similar  to  (\ref{EQN:weight.decoposition}), (\ref{EQN:weight.chordless.path}) and (\ref{EQN:weight.unique.path}) also hold for both  the correlation and the inflated correlation path weights.
\begin{proposition}\label{THM:prop.w.corr}
Under the conditions of Theorem~\ref{THM:jones.west} let $\Omega$ be the correlation matrix of $X_{V}$ and let $\Pt$ be a path between $x$ and $y$ in $\G$. Then it holds that:
\begin{enumerate}
    \item[(i)] $\w(\Pt, \Omega)=\w(\Pt, \Omega_{[P|\bar{P}]})\times \frac{\IF_{P}}{\sqrt{\IF_{x}^{\bar{P}}\IF_{y}^{\bar{P}}}}$;
    \item[(ii)] if $\Pt$ is a chordless path then $\w(\Pt, \Omega_{[P|\bar{P}]})=\rho_{xy\given \bar{P}}$ so that $\w(\Pt, \Omega)=\rho_{xy\given \bar{P}}\times \frac{\IF_{P}}{\sqrt{\IF_{x}^{\bar{P}}\IF_{y}^{\bar{P}}}}$;
    \item[(iii)] if $\Pt$ is the unique path between $x$ and $y$ in $\G$ then $\w(\Pt, \Omega)=\rho_{xy}$, so that
           \begin{align*}
                \rho_{xy}=\rho_{xy\given \bar{P}}\times \frac{\IF_{P}}{\sqrt{\IF_{x}^{\bar{P}}\IF_{y}^{\bar{P}}}}.
            \end{align*}
\end{enumerate}
\end{proposition}
\begin{proof}
Equation (i) can be shown as follows,
\begin{align}
\w(\Pt, \Omega) &= (\sigma_{xx}\sigma_{yy})^{-\frac{1}{2}}\w(\Pt, \Sigma)\label{EQN.cor.w.omega.01}\\
                &= (\sigma_{xx}\sigma_{yy})^{-\frac{1}{2}}\w(\Pt, \Sigma_{PP\given P})\IF_{P}\label{EQN.cor.w.omega.02}\\
                &= \left(\frac{\sigma_{xx\given\bar{P}}\sigma_{yy\given\bar{P}}}{\sigma_{xx}\sigma_{yy}}\right)^{\frac{1}{2}}
                   \w(\Pt, \Omega_{[P|\bar{P}]})\IF_{P}\label{EQN.cor.w.omega.03}\\
                &= \w(\Pt, \Omega_{[P|\bar{P}]})\times \frac{\IF_{P}}{\sqrt{\IF_{x}^{\bar{P}}
                   \IF_{y}^{\bar{P}}}},\label{EQN.cor.w.omega.04}
\end{align}
where (\ref{EQN.cor.w.omega.01}) and (\ref{EQN.cor.w.omega.03}) follow from Lemma~\ref{THM:phi.almost.scale.invariant} because  $\w(\Pt, \Omega)=(\sigma_{xx}\sigma_{yy})^{-1/2}\w(\Pt, \Sigma)$ and $\w(\Pt, \Omega_{[P|\bar{P}]})=(\sigma_{xx\given\bar{P}}\sigma_{yy\given\bar{P}})^{-1/2}\w(\Pt, \Sigma_{PP\given \bar{P}})$. Furthermore, (\ref{EQN.cor.w.omega.02}) is given by (\ref{EQN:weight.decoposition}), and the inflation factors in (\ref{EQN.cor.w.omega.04}) are computed from Lemma~\ref{THM:computation.of.IF}.

The identity (ii) follow from (i) because in a chordless path
$\w(\Pt, \Sigma_{PP\given \bar{P}})=\sigma_{xy\given\bar{P}}$ by (\ref{EQN:weight.chordless.path}), so that
$\w(\Pt, \Omega_{[P|\bar{P}]})
=(\sigma_{xx\given\bar{P}}\sigma_{yy\given\bar{P}})^{-1/2}\sigma_{xy\given\bar{P}}=\rho_{xy\given\bar{P}}$.
Finally, (iii) follows from (ii) because if $\Pt$ is the only path between $x$ and $y$ in $\G$, then $\w(\Pt, \Sigma)=\sigma_{xy}$ so that $\w(\Pt, \Omega)=(\sigma_{xx}\sigma_{yy})^{-1/2}\sigma_{xy}=\rho_{xy}$.
\end{proof}

The equality (iii) of Proposition~\ref{THM:prop.w.corr} was first given in
\citet{pmlr-v72-roverato18a}, where this factorization of the correlation coefficient was exploited to compare Gaussian tree models. They also noticed that this equality implies that $\mathrm{sgn}(\rho_{xy})=\mathrm{sgn}(\rho_{xy\given\bar{P}})$ and that $|\rho_{xy}|\geq |\rho _{xy\given \bar{P}}|$. Furthermore, as shown by \citet{choi2011learning} and \citet{zwiernik2015semialgebraic}, among others, when $\G$ is a tree, then the correlation $\rho_{xy}$ can be factorized over the edges of the unique path between $x$ and $y$ as $\rho_{xy}=\prod_{\{u,v\}\in \E(\Pt)} \rho_{uv}$. The latter implies that the partial correlation in both (ii) and (iii) can be factorized as $\rho_{xy\given \bar{P}}=\prod_{\{u,v\}\in \E(\Pt)} \rho_{uv\given \bar{P}}$, because $\G_{P}$ is a tree.

We now turn to the decomposition of inflated correlations and in this case we also give an explicit expression of the corresponding weights, which are easier to interpret with respect to the other types of weights considered so far.
\begin{theorem}\label{THM:prop.w.infcorr} Under the conditions of Theorem~\ref{THM:jones.west}, let $\infCORR{V}$ be the inflated correlation matrix of $X_{V}$ and let $\Pt$ be a path between $x$ and $y$ in $\G$. Then it holds that,
\begin{align}\label{EQN:weight.infomega.explicit}
\w(\Pt, \infCORR{V})= |\infCORR{V}_{PP}| \prod_{\{u.v\}\in \E(\Pt)}\rho_{uv\given V\setminus\{u,v\}};
\end{align}
and furthermore,
\begin{enumerate}
    \item[(i)] $\w(\Pt, \infCORR{V})=\w(\Pt, \infCORR{P}_{PP\given\bar{P}})\times \IF_{P}$ where
        \begin{align}\label{EQN:weight.partial.infomega.explicit}
          \w(\Pt, \infCORR{P}_{PP\given\bar{P}})
            =|\infCORR{P}_{PP\given\bar{P}}|
            \prod_{\{u.v\}\in \E(\Pt)}\rho_{uv\given V\setminus\{u,v\}};
        \end{align}
    \item[(ii)] if $\Pt$ is a chordless path $\w(\Pt, \infCORR{P}_{PP\given\bar{P}})=\varrho^{P}_{xy\given \bar{P}}$, so that $\w(\Pt, \infCORR{V})=\varrho^{P}_{xy\given \bar{P}}\times \IF_{P}$;

    \item[(iii)] if $\Pt$ is the unique path between $x$ and $y$ in $\G$, then $\w(\Pt, \infCORR{V})=\varrho^{V}_{xy}$ so that $\varrho^{V}_{xy}=\varrho^{P}_{xy\given \bar{P}}\times \IF_{P}$.
\end{enumerate}
\end{theorem}\begin{proof}
Equations (\ref{EQN:weight.infomega.explicit}) and (\ref{EQN:weight.partial.infomega.explicit}) follow from the direct computation of $\w(\Pt, \infCORR{V})$ and $\w(\Pt, \infCORR{V}_{PP\given\bar{P}})$, respectively, by recalling that $(\infCORR{V})^{-1}=I-R$ and $(\infCORR{P}_{PP\given\bar{P}})^{-1}=I_{PP}-R_{PP}$.

Equation (i) can be shown as follows,
\begin{align}
\w(\Pt, \infCORR{V}) &= (\kappa_{xx}\kappa_{yy})^{\frac{1}{2}}\w(\Pt, \Sigma)\label{EQN.cor.w.infomega.01}\\
                &= (\kappa_{xx}\kappa_{yy})^{\frac{1}{2}}\w(\Pt, \Sigma_{PP\given P})\IF_{P}\label{EQN.cor.w.infomega.02}\\
                &=  \w(\Pt, \infCORR{P}_{PP\given\bar{P}})\IF_{P}\label{EQN.cor.w.infomega.03}
\end{align}
where (\ref{EQN.cor.w.infomega.01}) and (\ref{EQN.cor.w.infomega.03}) follow from Lemma~\ref{THM:phi.almost.scale.invariant} because
$\Sigma_{PP\given \bar{P}}^{-1}=K_{PP}$ so that $\w(\Pt, \infCORR{V})=(\kappa_{xx}\kappa_{yy})^{\frac{1}{2}}\w(\Pt, \Sigma)$ and $\w(\Pt, \infCORR{P}_{PP\given\bar{P}})=(\kappa_{xx}\kappa_{yy})^{\frac{1}{2}}\w(\Pt, \Sigma_{PP\given \bar{P}})$ and (\ref{EQN.cor.w.infomega.02}) is given by (\ref{EQN:weight.decoposition}).
The identity (ii) follow from (i) because in a chordless path
$\w(\Pt, \Sigma_{PP\given \bar{P}})=\sigma_{xy\given\bar{P}}$ by (\ref{EQN:weight.chordless.path}), so that
$\w(\Pt, \infCORR{P}_{PP\given\bar{P}})=(\kappa_{xx}\kappa_{yy})^{\frac{1}{2}}\sigma_{xy\given\bar{P}}=\varrho_{xy\given\bar{P}}$.
Finally, (iii) follows from (ii) because if $\Pt$ is the only path between $x$ and $y$ in $\G$, then $\w(\Pt, \Sigma)=\sigma_{xy}$ so that $\w(\Pt, \infCORR{V})=(\kappa_{xx}\kappa_{yy})^{\frac{1}{2}}\sigma_{xy}=\varrho_{xy}$.
\end{proof}
Both correlations and inflated correlations can be regarded as normalized versions of covariances, and  Proposition~\ref{THM:prop.w.corr} and Theorem~\ref{THM:prop.w.infcorr} show that the behaviour of weights obtained from inflated correlations is more consistent with that of weights obtained from the decomposition of covariances. For example, the weights of covariances and inflated correlations involve the same type of inflation factor. Furthermore, the computation of (i) in Theorem~\ref{THM:prop.w.infcorr} from $\w(\Pt, \Sigma)$ involves two steps: the normalization of $\w(\Pt, \Sigma)$ to obtain $\w(\Pt, \infCORR{V})$ and then the factorization of the latter. The order in which these two steps are executed can be reversed, in the sense that one obtains the same result by first factorizing $\w(\Pt, \Sigma)$ and then normalizing $\w(\Pt, \Sigma_{PP\given\bar{P}})$. The same is not true for the weights computed from correlations.

We have focused on three different types of weights obtained from the decomposition of covariances, correlations and inflated correlations, respectively. Among these, we find that inflated correlations are of special interest because their value can be clearly interpreted. Consider the weights in (\ref{EQN:weight.infomega.explicit}) and (\ref{EQN:weight.partial.infomega.explicit}). These two quantities are computed  as the product of the partial correlations relative to the edges of the path, $\prod_{\{u.v\}\in \E(\Pt)}\rho_{uv\given V\setminus\{u,v\}}$,  multiplied by $|\infCORR{V}_{PP}|$ and $|\infCORR{P}_{PP\given\bar{P}}|$, respectively. The product of the partial correlations associated with the edges of the paths is perhaps the most intuitive measure to be associated with a path; however, we note that every partial correlation $\rho_{uv\given V\setminus \{u,v\}}$ is computed with respect to the distribution of $X_{\{u,v\}}$ adjusted for all the remaining variables in $X_{V}$. As a consequence, this product of partial correlations can be regarded as a ``maximally adjusted'' weight that needs to be embedded in the distribution where the path is considered. This is obtained by multiplying the product of partial correlations by the proper inflation factor, that is $|\infCORR{P}_{PP\given\bar{P}}|$ for $X_{P}|X_{\bar{P}}$ and $|\infCORR{V}_{PP}|$ for $X_{V}$; recall that, unlike $\Omega_{PP}$ that is computed on the marginal distribution of $X_{P}$, the computation of $\infCORR{V}_{PP}$ requires the joint distribution of $X_{V}$. We have discussed the interpretation of the determinants of inflated correlations in Section~\ref{SEC:inflation.factors}; here we remark that $|\infCORR{V}_{PP}|\geq |\infCORR{P}_{PP\given\bar{P}}|\geq 1$ and that from Theorem~\ref{THM:prop.w.infcorr} it follows that $\w(\Pt, \infCORR{V})/\w(\Pt, \infCORR{P}_{PP\given\bar{P}})|=|\infCORR{V}_{PP}|/|\infCORR{P}_{PP\given\bar{P}}|=\IF_{P}$. Finally, this interpretation of inflated correlation weights can be readily extended to interpret the meaning of the covariance weights, because by Lemma~\ref{THM:phi.almost.scale.invariant} it holds that $\w(\Pt, \Sigma)=\sqrt{\sigma_{xx\given V\setminus \{x\}}\sigma_{xx\given V\setminus \{x\}}}\times \w(\Pt, \infCORR{V})$.

\subsection{Single edge paths}
%---------------------------------------------------------
In the representation of a concentration graph it is of interest to identify a measure that may suitably encode the strength of the association represented by an edge of the graph. Commonly, this is done by associating the partial covariance $\sigma_{xy\given V\setminus \{x,y\}}$ with the edge $\{x,y\}$, which can then be normalized to obtain the partial correlation $\rho_{xy\given V\setminus \{x,y\}}$. On the other hand, \citet{roverato2017networked} observed that every edge of the graph is itself a path between its endpoints and therefore the corresponding covariance weights can provide an alternative measure of association for the edge. More formally, every edge is a chordless path and therefore, by (\ref{EQN:weight.chordless.path}),  the covariance weight of $\lp x, y\rp$ is given by $\w(\lp x, y\rp, \Sigma) = \sigma_{xy\given V\setminus\{x,y\}}\times \IF_{\{x,y\}}$. \citet{roverato2017networked} called this quantity a networked partial covariance and then they normalized it to obtain a networked partial correlation $\rho_{xy\given V\setminus\{x,y\}}\times \IF_{\{x,y\}}$. They also showed that networked partial correlations explain a larger fraction of variability of quantitative genetic interaction profiles in yeast than do marginal or partial correlations. Nevertheless, from Proposition~\ref{THM:prop.w.corr} if follows that $\rho_{xy\given V\setminus\{x,y\}}\times \IF_{\{x,y\}}\neq \w(\lp x, y\rp, \Omega)$ and more generally, to the best of our knowledge, networked partial correlations are not weights obtained from the decomposition of an association measure. On the other hand, Theorem~\ref{THM:prop.w.infcorr} suggests that inflated partial correlations represent an alternative way to normalize  networked partial covariances that is consistent with the path weight interpretation because it holds that,
\begin{align}\label{EQN:Phi.edges}
  \w(\lp x, y\rp, \infCORR{V})
  =\varrho^{\{x,y\}}_{xy\given V\setminus\{x,y\}}\times \IF_{\{x,y\}}.
\end{align}
We call this quantity a networked inflated partial correlation and remark that  networked partial correlations and networked inflated partial correlations are closely related because, by Lemma~\ref{THM:interpretation.entries.Rinv},
the inflated partial correlation in (\ref{EQN:Phi.edges}) is a simple one-to-one transformation that scales $\rho_{xy\given V\setminus\{x,y\}}$ to take values in $I\!\! R$,
\begin{align*}
\varrho^{\{x,y\}}_{xy\given V\setminus\{x,y\}}
=\frac{\rho_{xy\given V\setminus\{x,y\}}}
{1-\rho_{xy\given V\setminus\{x,y\}}^{2}}.
\end{align*}

\subsection{Upper and lower bounds}\label{SEC:normalized.weights}
%---------------------------------------------------------
Path weights, even when they are normalized quantities, do not take values in the interval $[-1, 1]$ and it is therefore of interest to compute their lower and upper bounds. To this aim it is useful to define the following function of $R$,
\begin{align*}%\label{EQN:normalized.weight}
\phi(\Pt, R)=|(I-R)_{\bar{P}\bar{P}}| \prod_{\{u,v\}\in \E(\Pt)}\rho_{uv\given V\setminus\{u,v\}},
\end{align*}
where $P=V(\Pt)$ and $\bar{P}=V\setminus P$; note that $\phi(\Pt, R)$ takes values
between $-1$ and $+1$.
\begin{proposition}\label{THM:normalized.weight}
Under the conditions of Theorem~\ref{THM:jones.west} let $\Gamma=\Delta\Sigma\Delta$ where $\Delta=\{\delta_{uv}\}_{u,v\in V}$ is a diagonal matrix with $\delta_{vv}\neq 0$ for all $v\in V$. Then for any path $\Pt$  between $x$ and $y$ in $\G$ it holds that,
\begin{align}\label{EQN:w.upper.lower.bounds}
|\w(\Pt, \Gamma)|\leq |\infCORR{V}| \sqrt{\delta_{xx}^{2}\,\sigma_{xx\given V\setminus\{x\}}\,\delta_{yy}^{2}\,\sigma_{yy\given V\setminus\{y\}}},
\end{align}
and, furthermore,
\begin{align*}
\frac{\w(\Pt, \Gamma)}{|\infCORR{V}| \sqrt{\delta_{xx}^{2}\,\sigma_{xx\given V\setminus\{x\}}\,\delta_{yy}^{2}\,\sigma_{yy\given V\setminus\{y\}}}}
= \phi(\Pt, R).
\end{align*}
\end{proposition}
\begin{proof}
We first notice that, because $|\infCORR{V}_{PP}|=|(I-R)_{\bar{P} \bar{P}}|/|I-R|$,
\begin{align*}
\w(\Pt, \infCORR{V})
= |\infCORR{V}_{PP}|\;\prod_{\{u,v\}\in \E(\Pt)}
        \rho_{uv\given V\setminus \{u,v\}}
=  \frac{\phi(\Pt, R)}{|I-R|}
= |\infCORR{V}|\; \phi(\Pt, R).
\end{align*}
Hence, we can apply Lemma~\ref{THM:phi.almost.scale.invariant} to obtain
\begin{align*}
\w(\Pt, \Gamma)=
\delta_{xx}\delta_{yy}\w(\Pt, \Sigma)
=\frac{\delta_{xx}\delta_{yy}\w(\Pt, \infCORR{V})}{\sqrt{\kappa_{xx}\,\kappa_{yy}}}
=\sqrt{\delta_{xx}^{2}\,\sigma_{xx\given V\setminus\{x\}}\,\delta_{yy}^{2}\,\sigma_{yy\given V\setminus\{y\}}}\;|\infCORR{V}| \phi(\Pt, R),
\end{align*}
as required. The upper and lower bounds in (\ref{EQN:w.upper.lower.bounds}) are an immediate consequence of the fact that $-1\leq \phi(\Pt, R)\leq 1$.
\end{proof}
The bounds given in (\ref{EQN:w.upper.lower.bounds}) cannot be improved  because they can actually be attained. This happens when $\phi(\Pt, R)=\pm{}1$, that is, when both $|(I-R)_{\bar{P}\bar{P}}|=1$ and $\rho_{uv\given V\setminus\{u,v\}}=\pm 1$ for all $\{u,v\}\in \E(\Pt)$. Clearly, it holds that $|(I-R)_{\bar{P}\bar{P}}|=1$ if and only if $R_{\bar{P}\bar{P}}=0$, i.e., the entries of $X_{\bar{P}}|X_{P}$ are mutually independent so that $\G_{\bar{P}}$ has no edges.

As far as the comparison of the weights of two paths is concerned, Proposition~\ref{THM:normalized.weight} shows that the weights of the paths joining two vertices have the same upper and lower bounds, i.e., they have the same range of possible values. On the other hand, two paths with different endpoints do not have, in general, the same upper and lower bounds. A remarkable exception is given by inflated correlations, for which all the weights have the same upper and lower bounds, regardless of the path endpoints.
\begin{corollary}\label{THM:infcor.upper.lower}
Under the conditions of Theorem~\ref{THM:jones.west} let $\infCORR{V}$ be the inflated correlation matrix of $X_{V}$ . Then for any path $\Pt$ between $x$ and $y$ in $\G$ it holds that,
\begin{align*}%\label{EQN:w.upper.lower.bounds.infCOR}
-|\infCORR{V}|\leq \w(\Pt, \infCORR{V})\leq |\infCORR{V}|,\qquad \frac{\w(\Pt, \infCORR{V})}{|\infCORR{V}|}=\phi(\Pt, R).
\end{align*}
\end{corollary}
\begin{proof}
This is an immediate consequence of Proposition~\ref{THM:normalized.weight} because in this case $\delta^{2}_{xx}=\kappa_{xx}=1/\sigma_{xx\given V\setminus\{x\}}$, and similarly for $\delta^{2}_{yy}$.
\end{proof}
We close this section by observing that the quantity  $\phi(\Pt, R)$ can be computed as the product of $\prod_{\{u,v\}\in \E(\Pt)}\rho_{uv\given V\setminus\{u,v\}}$, which is a function of $R_{PP}$, and of the determinant of $(I-R)_{\bar{P}\bar{P}}$. These two submatrices are variation independent in the sense that for any pair of positive-definite matrices with unit diagonal $(I-R)_{PP}$ and $(I-R)_{\bar{P}\bar{P}}$, there exists a positive-definite matrix $(I-R)$ with $(I-R)_{PP}$ and $(I-R)_{\bar{P}\bar{P}}$ as submatrices. This feature may be exploited in the estimation of $\phi(\Pt, R)$ from data, for instance when one is interested in the comparison of paths through  the ratio of their weights.

\section{Application to the analysis of dietary intake patterns}\label{SEC:application}
%---------------------------------------------------------
\citet{iqbal2016} introduced the use of dietary intake networks in the analysis of dietary patterns. Every vertex of the network represents the consumption of a given food group, and the edge structure shows how foods are consumed in relation to each other. More concretely, \citet{iqbal2016} applied graphical lasso \citep{friedman2008sparse} to a well-studied set of food intake data from $16\,340$ women and $10\,780$ men, to fit sex-specific concentration graph models. The number of food groups considered is 49, and for both women and men, the resulting network is made up of one major dietary network, called the principal intake network, and by a few disconnected small networks, consisting of similar food groups for women and men
\citep[see][Fig.~1 and 2]{iqbal2016}. The principal intake networks are given in Fig.~\ref{FIG:food.intake.networks} and they comprise 13 food groups for women and 12 for men, with \texttt{fried potatoes} consumption being not present in the latter. The two principal networks also differ with respect to the neighbours of the food group \texttt{legumes}. The estimates of the  partial correlations and of the   networked inflated partial correlations can be found in Table~\ref{TAB:food.intake.corr}. In this section we build on the analysis of  \citet{iqbal2016} in order to provide some examples on how the theory of path weights can be used to  address points of interest and to answer relevant questions.

\begin{figure}
\includegraphics[scale=0.9]{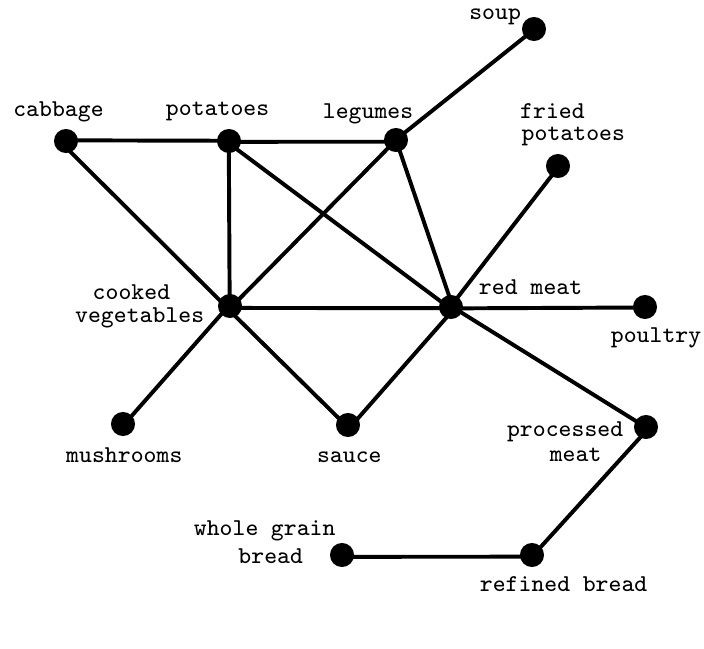}
\hfill
\includegraphics[scale=0.9]{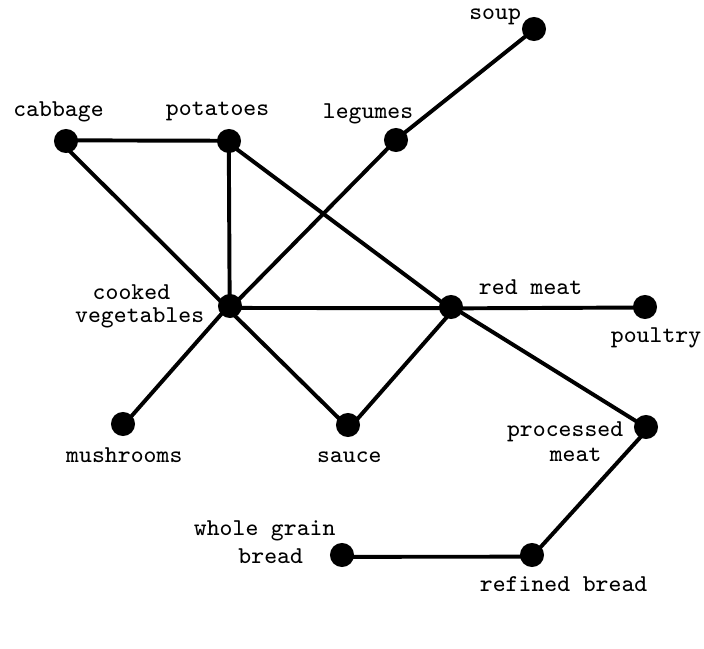}

\hspace*{8.5eM}(a)\hfill (b)\hspace*{8.5eM}

\caption{Principal dietary intake networks from \citet{iqbal2016}: (a) women, (b) men.}\label{FIG:food.intake.networks}
\end{figure}

\begin{table}[b]
\caption{Fitted values of partial correlations (PC) and networked inflated partial correlations (NIPC) for the edges of the intake networks in Fig.~\ref{FIG:food.intake.networks}.}\label{TAB:food.intake.corr}
\begin{center}
{\footnotesize
\begin{tabular}{|ll|r|r|r|r|ll|r|r|r|r|}
\hline
 & &\multicolumn{2}{|c|}{Women} & \multicolumn{2}{c|}{Men} & \multicolumn{2}{||c|}{~} &\multicolumn{2}{|c|}{Women} & \multicolumn{2}{c|}{Men}\\
\cline{3-6} \cline{9-12}
\multicolumn{2}{|c|}{\raisebox{0.5eM}{Edge}} &  \multicolumn{1}{c|}{PC} & \multicolumn{1}{c|}{NIPC} & \multicolumn{1}{c|}{PC} & \multicolumn{1}{c|}{NIPC} &
\multicolumn{2}{||c|}{\raisebox{0.5eM}{Edge}} &  \multicolumn{1}{c|}{PC} & \multicolumn{1}{c|}{NIPC} &  \multicolumn{1}{c|}{PC} & \multicolumn{1}{c|}{NIPC}\\
\hline
cooked veg.   & red meat      &$0.09$ & $0.16$ & $0.11$ & $0.19$   &\multicolumn{1}{||l}{legumes}        &red meat   &$0.07$ &$0.11$ & & \\
cooked veg.   &sauce          &$0.13$ & $0.18$ & $0.15$ & $0.21$  &\multicolumn{1}{||l}{fried pot.} &red meat   &$0.06$ &$0.09$ & &  \\
cooked veg.   &mushrooms      &$0.18$ & $0.23$ & $0.21$ & $0.28$  &\multicolumn{1}{||l}{poultry}        &red meat   &$0.29$ &$0.42$ & $0.28$ & $0.39$ \\
cooked veg.   &cabbage        &$0.22$ & $0.29$ & $0.29$ & $0.39$ &\multicolumn{1}{||l}{sauce}          &red meat    &$0.20$ &$0.29$ & $0.23$ & $0.33$ \\
cooked veg.   &potatoes       &$0.17$ & $0.23$ & $0.13$ & $0.19$   &\multicolumn{1}{||l}{potatoes}       &red meat  &$0.18$ &$0.28$ & $0.21$ & $0.31$ \\
cooked veg.   &legumes        &$0.17$ & $0.23$ & $0.12$ & $0.17$  &\multicolumn{1}{||l}{potatoes}       &cabbage    &$0.14$ &$0.17$ & $0.15$ & $0.19$ \\
whole bread         &refined br.  &-$0.37$&-$0.45$ & -$0.34$ & -$0.40$  &\multicolumn{1}{||l}{potatoes}       &legumes    &$0.10$  &$0.13$ & & \\
proc. meat      &refined br.  &$0.17$ & $0.23$ & $0.20$ & $0.26$  &\multicolumn{1}{||l}{soup}           &legumes    &$0.20$ &$0.23$ & $0.27$ & $0.30$ \\
proc. meat      &red meat       &$0.31$ & $0.46$ & $0.25$ & $0.37$  &\multicolumn{1}{||l}{~} & & & & &\\
\hline
\end{tabular}
}
\end{center}
\end{table}

An advantage of networks, with respect to traditional methods used in dietary pattern analysis, is that they provide the  association structure between food variables, thereby improving the understanding of the complexity of eating behaviours. One key issue that can be addressed using networks is the identification of food groups that play a central role in the eating behaviours. On the basis of degree centrality, i.e., the number of neighbours of vertices, \citet{iqbal2016} identified consumption of \texttt{red meat} and \texttt{cooked vegetables} as central to the dietary intake of both women and men and, in addition, also \texttt{legumes} and \texttt{potatoes} consumption were identified as central, but only for women. Degree centrality is a local measure that does not take into account the global structure of the network. On the other hand, one can consider betweenness-like centrality, which is based on some measure of relative relevance of paths passing trough a given vertex \citep{freeman1977set}. In concentration graph models, path weights constitute a natural way to quantify the relevance of a path, and accordingly the betweenness of a food group represented by a vertex $v$ with respect to vertices $x$ and $y$ can be computed as
\begin{align}\label{EQN:xyv-betweenness}
B_{xy}(v)=\frac{\sum_{\Pt\in \PT_{xy}; v\in\Pt}|\w(\Pt, \infCORR{V})|}{\sum_{\Pt\in \PT_{xy}}|\w(\Pt, \infCORR{V})|}
\end{align}
so that an overall measure of betweenness centrality of the vertex $v$ is given by $B(v)=\sum B_{xy}(v)$, where the sum is taken over all unordered pairs of vertices $x,y\in V$ with $x,y\neq v$. Because $B(v)$ scales with the number of pairs of vertices, the following normalization is usually performed:
\begin{align*}
\tilde{B}(v)=\frac{B(v)-B_{\min}}{B_{\max}-B_{\min}},
\end{align*}
where $B_{\min}=\min_{v\in V}\{B(v)\}$ and $B_{\max}=\max_{v\in V}\{B(v)\}$. In this way, it holds hat $\tilde{B}=1$ for the most central vertex and $\tilde{B}=0$ for the least central. It is important to remark that, although we use inflated correlation weights to calculate betweenness centrality, this quantity would not change if computed with any of the other types of weights considered in this paper, because it follows immediately from Lemma~\ref{THM:phi.almost.scale.invariant} that $B_{xy}(v)$ in (\ref{EQN:xyv-betweenness}) is invariant with respect to scale transformations of $\Sigma$. We also recall that betweenness centrality is more commonly computed by restricting attention to  shortest paths. Here, we follow the alternative approach of considering all paths \citep{borgatti2006graph}, because in this way the interpretation of $B_{xy}(v)$ in (\ref{EQN:xyv-betweenness}) is specially straightforward. Indeed, it can be shown from the network structures and the signs of partial correlations in Table~\ref{TAB:food.intake.corr} that the fitted distributions of interest  belong to the family of signed
totally positive distributions of order two for both women and men. As a consequence, $B_{xy}(v)$ in (\ref{EQN:xyv-betweenness}) is simply the proportion of $\hat{\varrho}^{V}_{xy}$, or equivalently of $\hat{\sigma}_{xy}$ and $\hat{\rho}_{xy}$, due to paths between $x$ and $y$ containing the vertex $v$. Table~\ref{TAB:betweenness} gives the values of betweenness and normalized betweenness for the principal intake networks of women and men. This analysis confirms the central role played by \texttt{red meat} and \texttt{cooked vegetables} consumption, but it also highlights that the consumption of \texttt{processed meat} is relevant with respect to centrality in both networks. Furthermore, \texttt{legumes} and \texttt{potatoes} seem to play a similar role in women and men, despite the higher numbers of neighbours they have in the network for women. Although comparison of alternative centrality measures is beyond the scope of this paper, it is worth mentioning that an analysis of betweenness centrality based on shortest paths gave results similar to those provided by Table~\ref{TAB:betweenness}, thus leading to comparable conclusions.

\begin{table}
\caption{Betweenness centrality, $B$, and normalized betweenness centrality, $\tilde{B}$, for the principal intake networks of women and men; variables are ordered according to the value of $\tilde{B}$ for women.}\label{TAB:betweenness}
\begin{center}
\begin{tabular}{|l|rr|rr|}
\hline
             &\multicolumn{2}{|c|}{Women} & \multicolumn{2}{c|}{Men}\\
\cline{2-5}
\multicolumn{1}{|c|}{\raisebox{0.5eM}{Food group}}       & \multicolumn{1}{c}{$B$}
& \multicolumn{1}{c|}{$\tilde{B}$\rule[0ex]{0ex}{2.5ex}}   & \multicolumn{1}{c}{$B$}  & \multicolumn{1}{c|}{$\tilde{B}$}\\
\hline
       red meat  &   45.63  &   1.00   &  33.72   &   1.00   \\
cooked vegetables&   24.35  &   0.53   &  31.94   &   0.95   \\
 processed meat  &   20.00  &   0.44   &  18.00   &   0.53   \\
        legumes  &   13.53  &   0.30   &  10.00   &   0.30   \\
       potatoes  &   11.72  &   0.26   &   7.44   &   0.22   \\
  refined bread  &   11.00  &   0.24   &  10.00   &   0.30   \\
          sauce  &    3.35  &   0.07   &   4.80   &   0.14   \\
        cabbage  &    1.23  &   0.03   &   2.19   &   0.07   \\
      mushrooms  &    0.00  &   0.00   &   0.00   &   0.00   \\
        poultry  &    0.00  &   0.00   &   0.00   &   0.00   \\
           soup  &    0.00  &   0.00   &   0.00   &   0.00   \\
    whole bread  &    0.00  &   0.00   &   0.00   &   0.00   \\
 fried potatoes  &    0.00  &   0.00   &          &          \\
\hline
\end{tabular}
\end{center}
\end{table}

The role played by \texttt{red meat} consumption in dietary patterns is specially important for the comprehension of eating behaviours and, as remarked by \citet{iqbal2016}, the strong positive partial correlation between \texttt{red meat} and \texttt{processed meat} is an interesting finding with possible implications in further analyses aimed at investigating the effect of meat consumption on health outcomes. This finding is confirmed by the values of networked inflated partial correlations in Table~\ref{TAB:food.intake.corr}, which take into account not only the partial association encoded by an edge, but also the relevance of the edge in terms of connections with the rest of the network. Indeed, for women, \texttt{red meat} and \texttt{processed meat} is the pair of variables with the strongest value of networked inflated partial correlation. We also note that there are three food groups relating to meat, specifically \texttt{red meat}, \texttt{processed meat} and \texttt{poultry}. In both networks these vertices form a path with three vertices and to investigate the relevance of this path within the two networks we computed the inflated correlation weights for all the paths on three vertices and compared them. This is meaningful because, by Corollary~\ref{THM:infcor.upper.lower}, unlike the other types of weights, it is possible to use the inflated correlation weight to compare paths with different endpoints. Interestingly, in both networks, the path $\lp\texttt{processed meat},\; \texttt{red meat},\;\texttt{poultry}\rp$ is the one with the strongest inflated correlation weight among all paths on three vertices, thereby providing evidence of the relevance of meat in food intake patterns.

One of the tasks of dietary pattern analysis is to understand how foods are consumed in relation to each other. From this perspective, one important structural difference between the two networks in Fig.~\ref{FIG:food.intake.networks} concerns the vertex \texttt{legumes}, which presents fewer connections with other vertices in the men's network than in the women's network. This affects, for example, the relationship between \texttt{soup} and \texttt{cooked vegetables} consumption. Indeed, for men this is entirely due to the path $\lp \texttt{soup},\;\texttt{legumes},\;\texttt{cooked vegetables}\rp$, whereas in the network for women there are nine paths between these two variables. Hence, a comparison of the two networks that only involves their edge structure would convey the information that the association structure of these two food groups is much more complex in women than in men. However, if we look at the weights of these 9 paths, we can easily see that in women the weight of the path $\lp \texttt{soup},\;\texttt{legumes},\;\texttt{cooked vegetables}\rp$ accounts for 81.4\% of the strength of the association between \texttt{soup} and \texttt{cooked vegetables}, while the remaining eight paths account for only 18.6\%. Therefore, the additional information provided by the values of path weights allows one to see that the path $\lp \texttt{soup},\;\texttt{legumes},\;\texttt{cooked vegetables}\rp$ is highly relevant also in women. From this observation, it follows that the way in which \texttt{soup} and \texttt{cooked vegetables} are consumed in relation to each other is more similar between men and women than would appear at first sight from a simple visual comparison of the  two networks. More generally, in the analysis of how two foods, $x$ and $y$ say, are consumed in relation to each other, path weights can be used to identify the most relevant paths between $x$ and $y$, that is, the most relevant dietary patterns. Finally, we remark that, also in this case, we do not need to choose which association measures to apply among covariance, correlation and inflated correlation, because all the types of weights considered in this paper would give the same results in this type of analysis.

\section*{Acknowledgments}
%------------------------------------------------------------
We gratefully acknowledge useful discussions with Beatrix Jones. We would like to thank the associate editor and two referees for their valuable comments which have led to improvements in the paper. We acknowledge the support of the Spanish MINECO/FEDER (TIN2015-71079-P), the Catalan AGAUR (SGR17-1020) and the European COST (CA15109). Alberto Roverato was supported by the Air Force Office of Scientific Research under award number FA9550-17-1-0039.
%
%------------------------------------------------------------------

\bibliographystyle{chicago}
\bibliography{pw-ref}
\end{document}